\pdfoutput=1
\documentclass[a4paper, 11pt, reqno]{amsart}

\usepackage{amssymb}
\usepackage{amsmath}
\usepackage{amsthm}
\usepackage[numbers]{natbib}
\usepackage{mathtools}
\usepackage{enumitem}
\usepackage{calc}
\usepackage{setspace}
\usepackage{graphicx}
\usepackage{hyperref}

\newcounter{dummy}
\numberwithin{dummy}{section}
\newtheorem{thm}[dummy]{Theorem}

\newtheorem{lem}[dummy]{Lemma}
\newtheorem{prop}[dummy]{Proposition}
\newtheorem{cor}[dummy]{Corollary}

\DeclareMathOperator{\id}{id}

\DeclareMathOperator{\GL}{GL}

\makeatletter
\renewcommand{\@biblabel}[1]{[#1]\hfill}
\makeatother

\begin{document}
\title{Generalized Vojta-R\'emond inequality}

\author{Gabriel A. Dill}
\address{Departement Mathematik und Informatik, Universit\"{a}t Basel, Spiegelgasse~1, CH-4051 Basel}
\email{gabriel.dill@unibas.ch}
\date{\today}

\begin{abstract}
Following and generalizing unpublished work of Ange, we prove a generalized version of R\'{e}mond's generalized Vojta inequality. This generalization can be applied to arbitrary products of irreducible positive-dimensional projective varieties, defined over the field of algebraic numbers, instead of powers of one fixed such variety. The proof runs closely along the lines of R\'{e}mond's proof.
\end{abstract}

\maketitle
\section{Introduction}
Let $m \geq 2$ be an integer and let $X_1, \hdots, X_m$ be a family of irreducible positive-dimensional projective varieties, defined over $\bar{\mathbb{Q}}$. We wish to extend R\'{e}mond's results of \cite{MR2233693} to the case of an algebraic point $x = (x_1,\hdots,x_m)$ in the product $X_1 \times \hdots \times X_m$. The following article is a further generalization of a generalization of these results by Thomas Ange. It draws heavily on a written account of this generalization by Ange \cite{A15}.

In \cite{D19}, we apply our generalized Vojta inequality to a relative version of the Mordell-Lang problem in an abelian scheme $\mathcal{A} \stackrel{\pi}{\to} S$, where $S$ is an irreducible variety and everything is defined over $\bar{\mathbb{Q}}$. In the problem, one fixes an abelian variety $A_0$, defined over $\bar{\mathbb{Q}}$, a finite rank subgroup $\Gamma \subset A_0(\bar{\mathbb{Q}})$ and an irreducible closed subvariety $\mathcal{V}Ê\subsetÊ\mathcal{A}$ and studies the points $p \in \mathcal{V}$ of the form $\phi(\gamma)$ for an isogeny $\phi: A_0 \to \mathcal{A}_{\pi(p)}$, $\mathcal{A}_{\pi(p)}$ denoting the fiber of the abelian scheme over $\pi(p)$, and $\gamma \in \Gamma$.

In this application, it is crucial that we allow the $X_i$ to lie in different fibers of the abelian scheme. If the abelian scheme $\mathcal{A}$ is constant, an analogue of the intended height bound has been obtained by von Buhren in \cite{MR3614529}. In his case, the generalized Vojta inequality from \cite{MR2233693}, where $X_1 = X_2 = \cdots = X_m = X$, was sufficient, however for our intended application it is necessary to allow the $X_i$ to be different.

Let us recall the hypotheses which come into play. We use (almost) the same notation as in \cite{MR2233693} and we refer to that article for the history of Vojta's inequality.

For an $m$-tuple $a = (a_1,\hdots,a_m)$ of positive integers, we write
\[Ê\mathcal{N}_a = \bigotimes_{i=1}^{m}{p_i^{\ast}\mathcal{L}_i^{\otimes a_i}},\]
where $\mathcal{L}_i$ is a fixed very ample line bundle on $X_i$ and $p_i: X_1 \times \hdots \times X_m \to X_i$ is the natural projection. We fix a non-empty open subset $U^{0} \subset X_1Ê\times \hdots \times X_m$ and relate $a$ to an irreducible projective variety $\mathcal{X}$, provided with an open immersion $U^{0} \subset \mathcal{X}$ and a proper morphism $\pi: \mathcal{X}Ê\to X_1 \timesÊ\hdots \times X_m$ such that $\pi|_{U^{0}} = \id_{U^{0}}$, as well as to a nef line bundle $\mathcal{M}$ on $\mathcal{X}$ which satisfies some further conditions, specified below.

We assume that there exists a very ample line bundle $\mathcal{P}$ on $\mathcal{X}$, an injection $\mathcal{P} \hookrightarrow \mathcal{N}_a^{\otimes t_1}$ which induces an isomorphism on $U^{0}$ and a system of homogeneous coordinates $\Xi$ for $\mathcal{P}$ which are (by means of the aforementioned injection) monomials of multidegree $t_1a$ in the homogeneous coordinates $W^{(i)}Ê\subset \Gamma(X_i,\mathcal{L}_i)$, fixed in advance (we denote $\pi^{\ast}\mathcal{N}_a$ also by $\mathcal{N}_a$ and identify $p_i^{\ast}W^{(i)}$ and $\pi^{\ast}p_i^{\ast}W^{(i)}$ with $W^{(i)}$). By (a system of) homogeneous coordinates for a very ample line bundle we mean the set of pull-backs of the homogeneous coordinates on some $\mathbb{P}^{N'}$ under a closed embedding into $\mathbb{P}^{N'}$ that is associated to that line bundle.

We also assume that there exists an injection $(\mathcal{P} \otimes \mathcal{M}^{\otimes -1}) \hookrightarrow \mathcal{N}_a^{\otimes t_2}$ which induces an isomorphism on $U^{0}$ and that $\mathcal{P} \otimes \mathcal{M}^{\otimes -1}$ is generated by a family $Z$ of $M$ global sections on $\mathcal{X}$ which are polynomials $P_1,\hdots,P_M$ of multidegree $t_2a$ in the $W^{(i)}$ such that the height of the family of coefficients of all these polynomials, seen as a point in projective space, is at most $\sum_{i} a_i \delta_i$. The height of any polynomial is defined by considering the family of its coefficients as a point in an appropriate projective space. On projective space, the height is defined as in Definition 1.5.4 of \cite{MR2216774} by use of the maximum norm at the archimedean places.

The integer parameters $t_1, t_2, M$ and the real parameters $\delta_1, \hdots, \delta_m$ (all at least $1$) are fixed independently of the triple $(a,\mathcal{X},\mathcal{M})$. This triple permits to define the following two notions of height for an algebraic point $xÊ\in U^{0}(\bar{\mathbb{Q}})$:
\[Êh_{\mathcal{M}}(x) = h(\Xi(x)) - h(Z(x)),\]
\[Êh_{\mathcal{N}_a}(x) = a_1h\left(W^{(1)}(x)\right) + \hdots + a_mh\left(W^{(m)}(x)\right).\]

Our goal is to prove an inequality among these two numbers under certain assumptions about the intersection numbers of $\mathcal{M}$. Let therefore $\thetaÊ\geq 1$ and $\omega \geq -1$ be two integer parameters and set (with $\omega' = 3 +Ê\omega$)
\[Ê\Lambda = \theta(2t_1u_0)^{u_0}\left(\max_{1Ê\leq i \leq m}{N_i}+1\right)\prod_{i=1}^{m}{\deg(X_i)},\]
\[Ê\psi(u) = \prod_{j=u+1}^{u_0}{(\omega'j+1)}, \]
\[Êc_1 = c_2 = \Lambda^{\psi(0)},\]
\[Êc^{(i)}_{3} = \Lambda^{2\psi(0)}(Mt_2)^{u_0}(h(X_i)+\delta_i) \quad (i=1,\hdots,m),Ê\]
where $u_0 = \dim(X_1)+\hdots+\dim(X_m)$, $N_i+1 = \#W^{(i)}$ and the degrees and heights are computed with respect to the embeddings given by the $W^{(i)}$. We use here the (normalized) height of a closed subvariety of projective space as defined in \cite{MR1260106} (via Arakelov theory) or \cite{MR1341770} (via Chow forms). The two definitions yield the same height by Th\'{e}or\`{e}me 3 of \cite{MR1144338}.

The following theorem therefore generalizes Th\'{e}or\`{e}me 1.2 of \cite{MR2233693}.

\begin{thm}\label{thm:vojta}
Let $x \in U^{0}(\bar{\mathbb{Q}})$ be an algebraic point and $(a,\mathcal{X},\mathcal{M})$ a triple as defined above. Suppose that, for every subproduct of the form $Y = Y_1Ê\times \hdots \times Y_m$, where $Y_i \subset X_i$ is a closed irreducible subvariety that contains $x_i$, we have the following estimate
\[Ê(\mathcal{M}^{\cdot \dim(Y)}Ê\cdotÊ\mathcal{Y}) \geq \theta^{-1}\prod_{i=1}^{m}{(\deg(Y_i))^{-\omega}a_i^{\dim(Y_i)}},\]
where $\mathcal{Y}$ denotes the closure of $\pi^{-1}(YÊ\cap U^{0})$ in $\mathcal{X}$. Then we have
\[Êh_{\mathcal{N}_a}(x) \leq c_1 h_{\mathcal{M}}(x)\]
if furthermore $c_2a_{i+1} \leq a_i$ for every $i < m$ and $c^{(i)}_{3} \leq h\left(W^{(i)}(x_i)\right)$ for every $iÊ\leq m$.
\end{thm}

The fact that for each $i$ there is a different constant $c^{(i)}_3$ is the main difference with Ange's work, where there is just one $\delta$ instead of $\delta_1,\hdots,\delta_m$ (in our set-up, $\delta$ can be taken as $\max_{1Ê\leq i \leq m}{\delta_i}$) and there is just one constant $c_3$ defined as
\[ \Lambda^{2\psi(0)}(Mt_2)^{u_0}\max\left\{\max_{1Ê\leq i \leq m}h(X_i),\delta\right\}.\]
The condition that $x_i$ has large height then reads $c_3 \leq h\left(W^{(i)}(x_i)\right)$. Ange's inequality is a direct generalization of R\'emond's inequality (up to the slightly different definitions of $\Lambda$ and $c_1$).

If we set $\delta = \max_{1Ê\leq i \leq m}{\delta_i}$, then the inequality $c^{(i)}_3 \leq h\left(W^{(i)}(x_i)\right)$ follows from $2c_3 \leq h\left(W^{(i)}(x_i)\right)$, so Theorem \ref{thm:vojta} really is a generalization of R\'emond's work (up to the factor $2$ and the slightly different definitions of $\Lambda$ and $c_1$). In the application in \cite{D19}, the fact that $c^{(i)}_3$ depends only on $h(X_i)$ and $\delta_i$ and not on $h(X_j)$ or $\delta_j$ ($j \neq i$) is crucial. Ange's version of the inequality is therefore not sufficient for the application.

Naturally, we follow the proof in \cite{MR2233693} very closely with some minor changes: Firstly, the term $12u\psi(u)$ that appears in the last equation of \cite{MR2233693} should be replaced by $4\omega'u\psi(u)$; that is why we do not use Lemme 5.4 of \cite{MR2233693} and define $\Lambda$ slightly differently. Secondly, Corollaire 5.1 of \cite{MR2233693} does not apply if $x_j^{(i)} = 0$, which means that Corollaire 3.2 of \cite{MR2233693} has to be made more precise. Thirdly, in the last inequality in the proof of Proposition 4.2 of \cite{MR2233693}, a term bounding the contribution of the archimedean places when the $P_i$ are raised to the $d$-th power is missing. Fourthly, the factor $8$ in the upper bound $8(N+1)D_i\log(N+1)D_i$ for $\log 2f_2(u_i,D_i)$ given in the proof of Proposition 5.3 of \cite{MR2233693} has to be increased. Fifthly, we had to impose that $\mathcal{M}$ is nef in order to be able to translate the lower bound on its top self-intersection number into a lower bound for the dimension of a space of global sections.

\section{Reduction to a minimal subproduct}

We first consider a subproduct $Y = Y_1Ê\times \dots \times Y_m$ of minimal total dimension $u = u_1 + \cdots + u_m$, satisfying the following conditions:
\begin{enumerate}[label=(\roman*)]
\item $x_i \in Y_i(\bar{\mathbb{Q}})$ for all $1 \leq i \leq m$;
\item $d_i \leq \deg(X_i)\Lambda^{\psi(u)-1}$ for all $1 \leq i \leq m$;
\item $\prod_{i=1}^{m}{d_i} \leq (\prod_{i=1}^{m}{\deg(X_i)})\Lambda^{\psi(u)-1}$;
\item $\sum_{i=1}^{m}{a_i(h_i+\delta_i)} \leq 2^{-1}\Lambda^{2\psi(u)}(Mt_2)^{u_0-u}\sum_{i=1}^{m}{(a_i(h(X_i)+\delta_i))}$,
\end{enumerate}
where $u_i = \dim(Y_i)$, $d_i = \deg(Y_i)$ and $h_i = h(Y_i)$ (in the projective embedding defined by $W^{(i)}$). Such a subproduct certainly exists, since $X_1 \times \hdotsÊ\times X_m$ satisfies these conditions. Furthermore, we have $u > 0$ since otherwise $Y = \{x\}$ and therefore
\[Ê\sum_{i=1}^{m}{a_ic^{(i)}_3} \leq h_{\mathcal{N}_a}(x) \leq \sum_{i=1}^{m}{a_ih_i} \leq \frac{1}{2}\sum_{i=1}^{m}{a_ic^{(i)}_3} .\]

We use the definition of an adapted projective embedding on p. 466 of \cite{MR2233693}. By Proposition 2.2 of \cite{MR2233693}, we may define an embedding adapted to the closed subvariety $Y_i$ of $X_i$ by putting
\[ V_j^{(i)} = \sum_{k=0}^{N_i}{M_{jk}^{(i)}W_k^{(i)}}\]
with $M^{(i)}Ê\in \GL_{N_i+1}(\mathbb{Q})$, where the coefficients of the matrix $M^{(i)}$ are integers and bounded by $\max\left(1,\frac{d_i}{2}\right)$ in absolute value, at least if $Y_i \neq \mathbb{P}^{N_i}$. If $Y_i = \mathbb{P}^{N_i}$, then the notion of an adapted projective embedding is not defined in \cite{MR2233693}, but we may set $V_j^{(i)} = W_j^{(i)}$ ($j=0,\hdots,N_i$) and check that all the assertions about adapted embeddings made in this article also hold true in this case.

We now prove the equivalent of Proposition 3.1 in \cite{MR2233693}, introducing
\[\Lambda_h = \sum_{i=1}^{m}{a_i(h_i+\delta_i+d_i(u_i+1)\log 2d_i(N_i+1))},\]
which we will prove to verify
\begin{equation}\label{eq:lambda-h-ineq}
\Lambda_h < \Lambda^{2\psi(u)}(Mt_2)^{u_0-u}\sum_{i=1}^{m}{(a_i(h(X_i)+\delta_i))} = \Lambda^{2\psi(u)-2\psi(0)}(Mt_2)^{-u}\sum_{i=1}^{m}{a_ic^{(i)}_3}.
\end{equation}
In order to show this inequality (given condition (iv) from above), it suffices to show that
\[Ê \sum_{i=1}^{m}{a_id_i(u_i+1)\log 2d_i(N_i+1)} < 2^{-1}\Lambda^{2\psi(u)}(Mt_2)^{u_0-u}\sum_{i=1}^{m}{(a_i(h(X_i)+\delta_i))}\]
or even $\sum_{i=1}^{m}{d_i(u_i+1)\log 2d_i(N_i+1)} < 2^{-1}\Lambda^{2\psi(u)}(Mt_2)^{u_0-u}$. But since $\log 2d_i(N_i+1) < 2d_i(N_i+1)$, it follows from (ii) that the left-hand side is at most
\[ (u+m)\Lambda^{2\psi(u)-2}\max_{1Ê\leqÊi \leq m}\{2\deg(X_i)^2(N_i+1)\}Ê< 2^{-1}\Lambda^{2\psi(u)}\]
and now the claim is obvious.

\begin{prop}\label{prop:minimality}
There does not exist any pair $(l,U)$ such that $1 \leq l \leq m$ and $U\left(V^{(l)}\right)$ is a homogeneous polynomial in the first adapted coordinates $V_0^{(l)}, \hdots, V_{u_l}^{(l)}$ satisfying
\begin{enumerate}[label=(\alph*)]
\item $U\left(V^{(l)}\right)(x_l)=0;$
\item $U$ is not the zero polynomial;
\item $\deg(U) \leq \Lambda^{\omega' u \psi(u)}$;
\item $a_l h(U) \leq \Lambda^{2\psi(u-1)-2\psi(u)}\left(\frac{Mt_2}{4d_l}\right)\Lambda_h$.
\end{enumerate}
\end{prop}

\begin{proof}
We assume the contrary and define $Y'_l$ as an irreducible component containing $x_l$ of the closed subvariety of $Y_l$ defined by the equation $U\left(V^{(l)}\right) = 0$ and we verify that the subproduct $Y'$ obtained by replacing $Y_l$ by $Y'_l$ in $Y$ contradicts the minimality of the latter. We have $Y' = Y_1' \times \dots \times Y_m'$ with $Y_i' = Y_i$ for all $i \neq l$. By (a), condition (i) holds for $Y'$. By (b) and the definition of an adapted embedding, $Y'$ is a proper subvariety of $Y$.

The polynomial $U\left(V^{(l)}\right)$ corresponds by means of $M^{(l)}$ to a polynomial $U'\left(W^{(l)}\right)$, where $\deg(U') = \deg(U)$ and
\[Êh(U') \leq h(U) + \deg(U)\log(N_l+1)\max\left(1,\frac{d_l}{2}\right)+\log\binom{\deg(U)+u_l}{\deg(U)}.\]
As
\[\binom{\deg(U)+u_l}{\deg(U)} = \prod_{i=1}^{\deg(U)}{\left(1+\frac{u_l}{i}\right)} \leq (1+u_l)^{\deg(U)},\]
it follows that
\begin{equation}\label{eq:height-ustrich}
h(U') \leq h(U) + \deg(U)\log d_l(N_l+1)(u_l+1).
\end{equation}

The (arithmetic as well as geometric) theorems of B\'{e}zout yield
\[ \deg(Y'_l) \leq \deg(U')d_l\]
and
\[Êh(Y'_l) \leq \deg(U')h_l + d_l\left(h(U')+\sqrt{N_l}\right).\]
For the arithmetic B\'{e}zout theorem, we use Th\'{e}or\`{e}me 3.4 and Corollaire 3.6 of \cite{MR1837829}, where the modified height $h_m$ used there can be bounded thanks to Lemme 5.2 of \cite{MR1838088}. Together with (c), the first line implies that $\deg(Y'_l) \leq d_l\Lambda^{\psi(u-1)-\psi(u)}$, since by definition $\psi(u-1) = (\omega'u+1)\psi(u)$. This shows that $Y'$ satisfies conditions (ii) and (iii).

From the second line together with \eqref{eq:height-ustrich}, (c) and (d), we deduce that
\begin{align*}
\sum_{i=1}^{m}{a_i(h(Y'_i)+\delta_i)} \leq d_la_lh(U)+d_la_l\Lambda^{\omega'u\psi(u)}\log d_l(N_l+1)(u_l+1)Ê\\
+ d_la_l\sqrt{N_l} + \Lambda^{\omega'u\psi(u)}\sum_{i=1}^{m}{a_i(h_i+\delta_i)} \leq d_la_lh(U) + 3\Lambda^{\omega'u\psi(u)}\Lambda_h\\
\leq \Lambda^{2\psi(u-1)-2\psi(u)}\left(\frac{Mt_2}{4}\right)\Lambda_h+3\Lambda^{\omega'u\psi(u)}\Lambda_h.
\end{align*}

Finally, we have $3\Lambda^{\omega'u\psi(u)} \leq \Lambda^{2\omega'u\psi(u)}\left(\frac{Mt_2}{4}\right) =\Lambda^{2\psi(u-1)-2\psi(u)}\left(\frac{Mt_2}{4}\right)$. It then follows from \eqref{eq:lambda-h-ineq} that $Y'$ satisfies condition (iv) as well and we get the desired contradiction.
\end{proof}

We proceed to deduce from this an equivalent of Corollaire 3.2 in \cite{MR2233693} (with a modification of the last assertion). Let us mention that by Lemme 2.3 of \cite{MR2233693}, there exist polynomial relations
\[ÊP_j^{(i)}\left(V_0^{(i)},\hdots,V_{u_i}^{(i)},V_j^{(i)}\right) = Q_j^{(i)}\left(V_0^{(i)},\hdots,V_{u_i}^{(i)},W_j^{(i)}\right) = 0 \mbox{ in }Ê\Gamma\left(Y_i,\mathcal{L}_i^{\otimes d_i}\right)\]
for all $1 \leq i \leq m$ and all $0 \leq j \leq N_i$. The polynomials $P_j^{(i)}(T)$ and $Q_j^{(i)}(T)$ are homogeneous of degrees $d_i$, monic in their last variable $T_{u_i+1}$ and equal to a power of an irreducible polynomial (we denote the corresponding exponent for $Q_j^{(i)}$ by $b_{i,j}$). Furthermore, we know from the same lemma that the height of the family $B_i$ of all the coefficients of the $P_j^{(i)}$ and the $Q_j^{(i)}$ for fixed $i$ (seen as a point in projective space) can be estimated from above as
\begin{equation}\label{eq:height-bi}
h(B_i) \leq h_i + d_i(u_i+1)\log d_i(N_i+1).
\end{equation}

\begin{cor}\label{cor:dacorollary}
For every index $1 \leq i \leq m$, we have that
\begin{enumerate}
\item the morphism $\rho_i: Y_i \to \mathbb{P}^{u_i}$, defined by the first adapted coordinates $V_0^{(i)}, \hdots, V_{u_i}^{(i)}$, is finite, surjective and \'{e}tale at $x_i \in Y_i(\bar{\mathbb{Q}})$;
\item $V_0^{(i)}(x_i) \neq 0$;
\item for every index $0 \leq jÊ\leq N_i$ such that $W_j^{(i)}Ê\neq 0$ in $\Gamma(Y_i,\mathcal{L}_i)$, we have
\[ W_j^{(i)}\frac{\partial^{b_{i,j}}Q_j^{(i)}}{\partial T_{u_i+1}^{b_{i,j}}}\left(1,\frac{V_1^{(i)}}{V_0^{(i)}},\hdots,\frac{V_{u_i}^{(i)}}{V_0^{(i)}},\frac{W_j^{(i)}}{V_0^{(i)}}\right)(x_i) \neq 0.\]
\end{enumerate}
\end{cor}

\begin{proof}
That the morphism $\rho_i$ is finite and surjective follows from the definition of adapted embeddings (see \cite{MR2233693}, Section 2.1). If one of the three assertions were not true, we could construct a pair $\left(i,U\left(V^{(i)}\right)\right)$ that would contradict Proposition \ref{prop:minimality}, with $\deg(U) \leq 2d_i^2$ and $h(U) \leq 6N_i d_i^3+2d_ih(B_i)$.

We refer to Corollaire 3.2 of \cite{MR2233693} for the proof -- in the case that $W_j^{(i)} \neq 0$, $W_j^{(i)}(x_i) = 0$ it suffices to take $U\left(V^{(i)}\right) = Q_j^{(i)}\left(V_0^{(i)},\hdots,V_{u_i}^{(i)},0\right)$. Note that $P^{(i)}_{u_i+1}$ is not only a power of an irreducible polynomial, but in fact irreducible, since its degree is equal to the degree of $Y_i$, which is also equal to the degree of any irreducible factor of $P^{(i)}_{u_i+1}$. Hence, its discriminant does not vanish identically. That the morphism $\rho_i$ is \'{e}tale at $x_i$ is proved in the same way as in the proof of Lemme 4.3 in \cite{MR1765539}.
\end{proof}

\section{Constructing a section of small height}

Following Section 4 of \cite{MR2233693}, we set
\[Ê\epsilon = \frac{1}{2u\theta}\frac{1}{(t_1m)^u}\prod_{i=1}^{m}{d_i^{-1-\omega}}\]
and define a family of sections $Z_d' \subset \Gamma(\mathcal{X},\mathcal{M}^{\otimes -d}Ê\otimes \mathcal{P}^{\otimes d} \otimes \mathcal{N}_a^{\otimesÊd \epsilon})$ of cardinality $M' = M(N_1+1)\cdots(N_m+1)$ for every $dÊ\in \epsilon^{-1}\mathbb{N} \subset \mathbb{N} = \{1,2,3,\hdots\}$ by
\[ÊZ_d' = \left\{Ê\zeta^{\otimes d}Ê\otimes \left(W_{j_1}^{(1)}\right)^{\otimes d\epsilon a_1} \otimes \cdotsÊ\otimes \left(W_{j_m}^{(m)}\right)^{\otimes d\epsilon a_m}; \zeta \in Z, W_{j_i}^{(i)} \in W^{(i)}\right\}.\]
The proof of Proposition 4.1 of \cite{MR2233693} then goes through without any major modifications (given that $\mathcal{M}$ is nef, see below). It yields a natural number $d_0$ that we choose sufficiently large so that for each natural number $dÊ\geq d_0$ there exists a basis of $\Gamma(\mathcal{Y},\mathcal{P}^{\otimes d})$ that consists of monomials of degree $d$ in the elements of $\Xi$. We obtain the following equivalent of Proposition 4.2 in \cite{MR2233693}.

\begin{prop}\label{prop:siegel}
For $d \in \epsilon^{-1}\mathbb{N} \cap d_0\mathbb{N}$, we write $\mathcal{Q}_d = \mathcal{M}^{\otimes d} \otimes \mathcal{N}_a^{\otimes -d\epsilon}$ and fix a basis of $\Gamma(\mathcal{Y},\mathcal{P}^{\otimes d})$ that consists of monomials in the sections $\Xi$ of degree $d$.

Then there exists a section $0 \neq s \in \Gamma(\mathcal{Y},\mathcal{Q}_d)$ such that the height of $s$, defined as the height of the family of coefficients of the sections $sÊ\otimes Z_d'$ with respect to the fixed basis, seen as a point in projective space, satisfies
\[Êh(s) \leq \frac{2M'd}{u\epsilon}(t_1+2t_2+\epsilon)\Lambda_h+o(d).\]
\end{prop}

\begin{proof}
The dimension estimate
\[ \dim \Gamma(\mathcal{Y},\mathcal{Q}_d)Ê\geq \frac{d^u}{4\theta u!}\prod_{i=1}^{m}{d_i^{-\omega}a_i^{u_i}}+O(d^{u-1})\]
given in Proposition 4.1 of \cite{MR2233693} is still valid, since the intersection numbers are formally the same. Here, we need however that $\mathcal{M}$ is nef in order to translate the lower bound for its top self-intersection number into a lower bound for the dimension of a space of global sections through the asymptotic Riemann-Roch theorem (see \cite{MR1440180}, Theorem VI.2.15).

In the Faltings complex on $\mathcal{Y}$ defined by the family $Z_d'$ of cardinality $M'$
\[Ê0 \to \mathcal{Q}_d \to \left(\mathcal{P}^{\otimes d}\right)^{\oplus M'} \to \left(\mathcal{N}_a^{\otimes d(t_1+t_2+\epsilon)}\right)^{\oplus (M')^2}\]
the image of $\Gamma(\mathcal{Y},\mathcal{Q}_d)$ in $F = \Gamma(\mathcal{Y},\mathcal{P}^{\otimes d})^{M'}$ coincides with the kernel of a family of linear forms in the coordinates with respect to the fixed basis. This family can be chosen such that the coefficients of the linear forms lie in a number field that is independent of $d$ and the height of the set of all coefficients, seen as a point in projective space, is at most $d(t_1+2t_2+\epsilon)\Lambda_h + o(d)$: in order to show this, we follow the proof of Proposition 4.2 of \cite{MR2233693} by applying Lemme 2.5 in \cite{MR2233693} with $n_i = N_i$ and use \eqref{eq:height-bi} to bound $h(B_i)$. Note that when estimating $h(P_1^d,\hdots,P_M^d)$ as in the proof of Proposition 4.2, one obtains by well-known height estimates an upper bound of
\[Êd\sum_{i=1}^{m}{a_i\delta_i}+dt_2\sum_{i=1}^{m}{a_i\log(N_i+1)}\]
(the second summand, coming from the archimedean places, is missing in \cite{MR2233693}).

Furthermore, the injection $\mathcal{P}^{\otimes d} \hookrightarrow \mathcal{N}_a^{\otimes dt_1}$ yields that
\[Ê\dim F \leq M'\prod_{i=1}^{m}{\frac{d_i}{u_i!}(dt_1a_i)^{u_i}}+o(d^u)\]
and so $\log \dim F = o(d)$. Hence, the Dirichlet exponent of the system can be estimated as
\[Ê\frac{\dim F}{\dim \Gamma(\mathcal{Y},\mathcal{Q}_d)} \leq \frac{2M'}{u\epsilon} + o(1)\]
and the proposition follows from the Siegel lemma (Lemme 2.6 in \cite{MR2233693}).
\end{proof}

\section{The index is small}

We now replace $\mathcal{Y}$ by a sufficiently small open subset of $\mathcal{Y}$ that contains $x$. According to Corollary \ref{cor:dacorollary}, we can in particular assume that each section $V_0^{(i)}$ vanishes nowhere on this subset and suppose that the sheaf of differentials $\Omega_{\mathcal{Y}/\bar{\mathbb{Q}}}$ is generated by the differentials of the $V_j^{(i)}/V_0^{(i)}$ ($i=1,\hdots,m$, $1 \leq j \leq u_i$). We can furthermore suppose that $\mathcal{P}$, $\mathcal{M}$ and $\mathcal{N}_a$ all can be trivialized over this subset.

We fix an isomorphism $\mathcal{Q}_d \simeq \mathcal{O}_{\mathcal{Y}}$ and consider the index $\sigma$ (as defined in Section 5.2 of \cite{MR2233693}) of the section $s_d \in \Gamma(\mathcal{Y},\mathcal{Q}_{d})$ that was constructed in the preceding proposition with respect to the weight $dt_1a$ in $x$.

\begin{lem}
With notations as above, we have
\[Ê\sigma \leq (4t_1\max_{i}{d_i(N_i+1)})^{-1}\epsilon\]
for $d \inÊ\epsilon^{-1}\mathbb{N}Ê\cap d_0\mathbb{N}$ sufficiently large.
\end{lem}
\begin{proof}
We assume that the inequality is false and derive a contradiction. We can estimate
\begin{align*}
\sigma\prod_{i=1}^{m}{d_i^{-1}} \geq  (4t_1\max_{i}{d_i(N_i+1)})^{-1}\epsilon\prod_{i=1}^{m}{d_i^{-1}} \\
\geq (8u\theta t_1^{u+1}m^u\max_{i}(N_i+1))^{-1}\prod_{i=1}^{m}{d_i^{-\omega'}}.
\end{align*}
It then follows from (iii) that
\[ \sigma\prod_{i=1}^{m}{d_i^{-1}} \geq (8u\theta t_1^{u+1}m^u\max_{i}(N_i+1))^{-1}\prod_{i=1}^{m}{(\deg X_i)^{-\omega'}}\Lambda^{-\omega'(\psi(u)-1)}\]
and hence $\sigma\prod_{i=1}^{m}{d_i^{-1}}  \geq \sigma_0 = m\Lambda^{-\omega'\psi(u)}$.

Then, we can construct a non-zero multihomogeneous polynomial $G(V)$ of multidegree $dt_1(d_1\cdots d_m)a$ in the adapted coordinates $V_j^{(i)}$, $0 \leq j \leq u_i$, of height bounded by
\[ h(G) \leq (d_1\cdots d_m)\left(h(s_d)+dt_1\sum_{i=1}^{m}{a_i(h(B_i)+\log(2(u_i+1)))}\right) + o(d)Ê\]
and of index at least $\sigma$ in $\rho(x)$ with respect to the weight $dt_1a$, where $\rho = (\rho_1 \circ p_1|_Y,\hdots,\rho_m \circ p_m|_Y)$.

For this, we choose $\zeta' \in Z_d'$ which does not vanish at $x$. We write $s_d \otimes \zeta' = \alpha \left(\left(V_0^{(1)}\right)^{\otimes dt_1a_1}\otimes\cdots\otimes \left(V_0^{(m)}\right)^{\otimes dt_1a_m}\right)$, where $\alpha$ is a polynomial in the $W^{(i)}_j/V^{(i)}_{0}$ with coefficients in $\bar{\mathbb{Q}}$. Consider the norm $N(\alpha)$ of $\alpha$ with respect to the field extension $\bar{\mathbb{Q}}(\mathcal{Y})/L$, where $L$ is the subfield of $\bar{\mathbb{Q}}(\mathcal{Y})$ generated by the $V_j^{(i)}/V_0^{(i)}$ ($j=1,\hdots,u_i$, $i=1,\hdots,m$). We can take
\[ G = \left(V^{(1)}_{0}\right)^{dt_1(d_1\cdots d_m)a_1} \hdots \left(V^{(m)}_{0}\right)^{dt_1(d_1\cdots d_m)a_m} N(\alpha).\]

On the one hand, this is a quotient of multihomogeneous elements of $R = \bar{\mathbb{Q}}[V_j^{(i)}; 1Ê\leq i \leq m, 0 \leq j \leq u_i]$. As such it has a multidegree, which is exactly $dt_1(d_1\cdots d_m)a$. On the other hand, it is the norm of the multihomogenization of $\alpha$, which is a multihomogeneous polynomial in the $W_j^{(i)}$, with respect to the field extension
\[ \tilde{L}/\bar{\mathbb{Q}}\left(V_j^{(i)}; 1Ê\leq i \leq m, 0 \leq j \leq u_i\right),\]
where $\tilde{L}$ is the fraction field of the multihomogeneous coordinate ring of $Y_1 \times \cdots \times Y_m \hookrightarrow \mathbb{P}^{N_1} \times \cdotsÊ\times \mathbb{P}^{N_m}$. As such $G$ is integral over $R$ and therefore lies in $R$. So $G$ is in fact a multihomogeneous polynomial of the desired multidegree.

Note that $\beta = N(\alpha)\alpha^{-1} = \prod_{\tau \neq \id}{\tau(\alpha)}$ lies in $\bar{\mathbb{Q}}(\mathcal{Y})$ and is integral over $\mathcal{O}_{\mathbb{P}^{u_1} \timesÊ\cdots \times \mathbb{P}^{u_m},\rho(x)}$ because $\alpha$ is. Here, $\tau$ runs over the embeddings of $\bar{\mathbb{Q}}(\mathcal{Y})$ into a normal closure of the extension $\bar{\mathbb{Q}}(\mathcal{Y})/L$. Hence, $\beta$ is integral over $\mathcal{O}_{\mathcal{Y},x}$. As $\rho$ is \'etale at $x$, this local ring is normal and hence contains $\beta$. So the index of $N(\alpha)$ in $x$ (or equivalently, the index of $G$ in $\rho(x)$) is greater or equal than the index of $\alpha$ in $x$. For the bound for $h(G)$, see Lemme 5.5 of \cite{MR2233693}.

We can then apply Th\'{e}or\`{e}me 5.6 (Faltings' product theorem) of \cite{MR2233693} with the value of $\sigma_0$ above and obtain in this way a contradiction with Proposition \ref{prop:minimality}. The hypotheses of the theorem are satisfied, since
\[Ê\frac{a_i}{a_{i+1}} \geq c_2 \geq \left(\frac{m}{\sigma_0}\right)^u \geq (2u^2)^{u^2}\]
and $G$ has index at least $\sigma$ with respect to the weight $dt_1a$ in $\rho(x)$, hence has index at least $\sigma\prod_{i=1}^{m}{d_i^{-1}}Ê\geq \sigma_0$ with respect to the weight $dt_1(d_1\cdots d_m)a$ in $\rho(x)$.

We obtain a pair $(l,U)$ with $U\left(V^{(l)}\right)(x_l) = 0$, $U$ non-zero, $\deg(U) \leq \left(\frac{m}{\sigma_0}\right)^{u} = \Lambda^{\omega' u \psi(u)}$ and
\begin{align*}
a_lh(U) \leq u_l\left(\frac{m}{\sigma_0}\right)^u\left(\frac{h(G)}{dt_1d_1\cdots d_m}+\sum_{i=1}^{m}{a_i(u_i\log(u_i+1)+\log 2)}\right) \\
+ a_l\left(\frac{m}{\sigma_0}\right)^{u}(u_l+1)\log\left(\frac{m}{\sigma_0}\right)^u(u_l+1)+a_l\log\binom{\deg(U)+u_l}{u_l}+o(1).
\end{align*}
Here, we used that the height of projective $n$-space is bounded from above by $n\log(n+1)$.

After some simplification and by using that $u_l \geq 1$ (which is a consequence of the product theorem) and $\frac{m}{\sigma_0} = \Lambda^{\omega'\psi(u)}$, we deduce that
\begin{align*}
a_lh(U) \leq u_l\Lambda^{\omega'u\psi(u)}\left(\frac{h(s_d)}{dt_1}+2\Lambda_h+2a_l\log2\left(\frac{m}{\sigma_0}\right)^u(u_l+1)\right)+o(1) \\
\leq u_l\Lambda^{\omega'u\psi(u)}\left(\frac{2M'}{u\epsilon}(1+2t_2+3\epsilon)\Lambda_h + 2a_l \log \left(\frac{m}{\sigma_0}\right)^u\right) + o(1).
\end{align*}
For the last inequality, we used that $2a_l\log2(u_l+1)Ê\leq 2\Lambda_h$ and $\frac{2M'}{u}Ê\geq 2$.

We can now estimate
\[Ê2a_lu_l\log\left(\frac{m}{\sigma_0}\right)^u \leq 2\Lambda_h\omega'u\psi(u)\log \Lambda \leqÊ\Lambda_h\Lambda^{(\omega'u-1)\psi(u)},\]
since $a_lu_l \leq \Lambda_h$, $2\omega'u\psi(u)\log \Lambda \leq \Lambda^{\frac{1}{2}\omega'u\psi(u)}$ and $\frac{1}{2}\omega'u \leq \omega'u-1$.
Thanks to (iii), we can bound the first term as
\begin{align*}
\frac{2M'u_l}{u \epsilon} \leq \frac{2M'}{\epsilon} \leq 2M\max_{i}{(N_i+1)}^m(2u\theta)(t_1m)^u\prod_{i=1}^{m}{\deg(X_i)^{1+\omega}}\Lambda^{(1+\omega)(\psi(u)-1)} \\
\leq M\Lambda^{m}\Lambda^{(1+\omega)\psi(u)} \leq M\Lambda^{(\omega'u-1)\psi(u)},
\end{align*}
since $m+(2+\omega)\psi(u) \leq \omega'u\psi(u)$. This last inequality follows from $m \leq u\psi(u)$.

Combining these inequalities with the one above, we obtain that
\[Êa_lh(U) \leq \Lambda^{(2\omega'u-1)\psi(u)}M\Lambda_h(2+2t_2+3\epsilon) + o(1) \leq \Lambda^{2\psi(u-1)-2\psi(u)}\left(\frac{Mt_2}{4d_l}\right)\Lambda_h,\]
where we used that $2+2t_2+3\epsilon \leq 5t_2 \leq \frac{\Lambda^{\psi(u)} t_2}{4 d_l}$ by (ii). We could get rid of the $o(1)$, since for example this last inequality is in fact strict. Thus, we have found a contradiction with Proposition \ref{prop:minimality}.
\end{proof}

\section{Finishing the proof}

We now have established that the section $s_d \in \Gamma(\mathcal{Y},\mathcal{Q}_d)$ given by Proposition \ref{prop:siegel} has index (in $x$ and with respect to the weight $dt_1a$) bounded as
\[Ê\sigma \leq (4t_1\max_{i}{d_i(N_i+1)})^{-1}\epsilon.\]
We write $D$ for a differential operator associated to that index and finish the proof of Theorem \ref{thm:vojta} by considering the following height
\begin{align*}
-h_{\mathcal{Q}_d}(x) = dh(Z(x))-dh(\Xi(x))+d\epsilon\sum_{i=1}^{m}{a_ih(W^{(i)}(x))} \\
= h(Z_d'(x))-dh(\Xi(x)).
\end{align*}
By definition, there exists such a $D$ with $D(s_d)(x) \neq 0$ and we have $D'(s_d)(x) = 0$ for every operator $D'$ of index $\sigma' < \sigma$, hence by the product formula
\[Êh(Z_d'(x)) = h\left((D(s_d)Ê\otimes Z_d')(x)\right) = h \left((D(s_d \otimes \zeta')(x))_{\zeta' \in Z_d'}\right).\]
In order to define the right-hand side, one has to fix an isomorphism $\mathcal{P}^{\otimes d} \simeq \mathcal{O}_{\mathcal{Y}}$. The right-hand side is however independent of the choice of isomorphism, precisely since $D$ is an operator associated to the index of $s_d$.

Let us recall that the sections $s_d \otimes \zeta' \in \Gamma(\mathcal{Y},\mathcal{P}^{\otimes d})$ are homogeneous polynomials of degree $d$ in the sections $\Xi$ and that $\log \dim \Gamma(\mathcal{Y},\mathcal{P}^{\otimes d}) = o(d)$. Furthermore, the sections $\Xi$ themselves are monomials of multidegree $t_1a$ in the coordinates $W^{(i)}$. Hence, the right choice of isomorphism shows that
\[Êh(Z_d'(x)) \leq h(s_d) + h\left((D(\xi^\nu)(x))_{\xi^\nu}\right) + o(d), \]
where $\xi^\nu$ runs over the monomials of degree $d$ in the sections $\Xi$ (seen as monomials of multidegree $dt_1a$ in the $W^{(i)}$) divided by appropriate products of the $V^{(i)}_0$ ($i=1,\hdots,m$).

We can estimate the height of the $D(\xi^\nu)(x)$ by using Leibniz' formula as well as Corollaire 5.1 and Lemme 5.2 of \cite{MR2233693} (corrected). For $1 \leq i \leq m$ and $l = (l_1,\hdots,l_{u_i}) \in (\mathbb{N}Ê\cup \{0\})^{u_i}$, we define the operator
\[Ê\partial^{i,l} = \prod_{j=1}^{u_i}{\frac{1}{l_j!}\left(\partial_{V_j^{(i)}/V_0^{(i)}}\right)^{l_j}}: \mathcal{O}_{\mathcal{Y},x} \to \mathcal{O}_{\mathcal{Y},x}.\]
If $w = (w_1,\hdots,w_k) \in (\mathbb{N}\cup\{0\})^k$ is a multi-index, we write $|w| = w_1 + \cdots + w_k$.

\begin{lem}
Let $1 \leq i \leq m$ be an integer and let $K$ be a number field that contains the coordinates $\left(W_j^{(i)}/V_0^{(i)}\right)(x)$, the families $B_i$ and the products
\[ c_i = \prod_{j}{\left(\frac{\left(W_j^{(i)}/V_0^{(i)}\right)^{b_{i,j}}}{b_{i,j}!}\frac{\partial^{b_{i,j}}Q_j^{(i)}}{\partial T_{u_i+1}^{b_{i,j}}}\left(1,\frac{V_1^{(i)}}{V_0^{(i)}},\hdots,\frac{V_{u_i}^{(i)}}{V_0^{(i)}},\frac{W_j^{(i)}}{V_0^{(i)}}\right)(x_i)\right)^{\frac{1}{b_{i,j}}}},\]
where $j$ runs over the indices satisfying $W_j^{(i)} \neq 0$ in $\Gamma(Y_i,\mathcal{L}_i)$. Then for every place $v$ of $K$ and every multi-index $l \in (\mathbb{N} \cup \{0\})^{u_i}$, we have
\[ \left|\partial^{i,l}\left(W_j^{(i)}/V_0^{(i)}\right)(x)\right|_v \leq \left|\left(W_j^{(i)}/V_0^{(i)}\right)(x)\right|_v(|c_i|_v^{-2}C_{i,v})^{|l|}Ê\]
with
\[ÊC_{i,v} = 2^{-\epsilon_v}\left((d_i(N_i+1))^{6d_i \epsilon_v}\max_{b \in B_i}{|b|_v}\max_{0 \leq k \leq N_i}{\left|\left(W_k^{(i)}/V_0^{(i)}\right)(x)\right|_v^{d_i}}\right)^{2(N_i+1)}\]
and $\epsilon_v = 1$ if $v$ is infinite, $0$ if $v$ is finite.
\end{lem}

\begin{proof}
Recall that by Corollary \ref{cor:dacorollary}, the number $c_i \in \bar{\mathbb{Q}}\backslash\{0\}$ is well defined (up to the choice of the roots which can be made arbitrarily).

If $W_j^{(i)} = 0$ in $\Gamma(Y_i,\mathcal{L}_i)$, the derivative $\partial^{i,l}\left(W_j^{(i)}/V_0^{(i)}\right)$ is zero and the inequality holds. Otherwise, we may apply Corollaire 5.1 of \cite{MR2233693} and follow the proof of Lemme 5.2 of \cite{MR2233693} with $N=N_i$, using at the end that
\[Ê2f_2(u_i,d_i) = 2(2u_i+4)^{1+\frac{3}{2}d_i}d_i\binom{d_i+u_i}{u_i}^{2(N_i+2)}\binom{d_i+1}{2}^2(2(N_i+1)d_i)^{2(N_i+1)d_i}\]
is bounded from above by
\begin{align*}
2^{2+\frac{3}{2}d_i+2(N_i+1)d_i}(N_i+2)^{1+\frac{3}{2}d_i}(N_i+1)^{2(N_i+2)d_i}d_i^5((N_i+1)d_i)^{2(N_i+1)d_i} \\
\leq (N_i+1)^{N_i+1+\frac{3}{2}d_i+\frac{\log 3}{\log 2}(1+\frac{3}{2}d_i)+5(N_i+1)d_i}d_i^5((N_i+1)d_i)^{2(N_i+1)d_i} \\
\leq (N_i+1)^{2(N_i+1)+2(N_i+1)d_i+5(N_i+1)d_i}d_i^5((N_i+1)d_i)^{2(N_i+1)d_i}  \\
\leq (d_i(N_i+1))^{12d_i(N_i+1)}.
 \end{align*}
\end{proof}
For $D = \prod_{i=1}^{m}{\partial^{i,\kappa_i}}$, we obtain the following bound (cf. the proof of Proposition 5.3 in \cite{MR2233693})
\[Ê|D(\xi^{\nu})(x)|_v \leq |\xi^{\nu}(x)|_v \prod_{i=1}^{m}{2^{(|\kappa_i|+dt_1u_ia_i)\epsilon_v}(|c_i|_v^{-2}C_{i,v})^{|\kappa_i|}}\]
and hence thanks to the product formula for the $c_i$
\begin{align*}
h\left((D(\xi^\nu)(x))_{\xi^\nu}\right) \leq dh(\Xi(x)) + \sum_{i=1}^{m}{2(N_i+1)|\kappa_i|(h(B_i) + d_ih(W^{(i)}(x)))} \\
+ \sum_{i=1}^{m}{(dt_1u_ia_i\log 2 + 12d_i(N_i+1)|\kappa_i|\log d_i(N_i+1))}.
\end{align*}

We proceed with bounding
\[Ê|\kappa_i| \leq dt_1a_i\sigma \leq (4(N_i+1)d_i)^{-1}d\epsilon a_i,\]
which implies that $\sum_{i=1}^{m}2d_i(N_i+1)|\kappa_i|h(W^{(i)}(x)) \leq \frac{d\epsilon}{2}h_{\mathcal{N}_a}(x)$. Together with \eqref{eq:height-bi}, the bound also implies that
\begin{align*}
\sum_{i=1}^{m}{2(N_i+1)|\kappa_i|(h(B_i)+6d_i\log d_i(N_i+1))} \\
\leq d\epsilon\sum_{i=1}^{m}{a_i \left(\frac{h_i+d_i(u_i+1)\log d_i(N_i+1)}{2d_i} + 3\log d_i(N_i+1)\right)}Ê\leq 4d\epsilon \Lambda_h.
\end{align*}

Finally we know that $\sum_{i=1}^{m}{u_ia_i\log2} \leqÊ\Lambda_h$ and putting all these estimates together we get
\[Ê\epsilon h_{\mathcal{N}_a}(x) - h_{\mathcal{M}}(x) = -\frac{h_{\mathcal{Q}_d}(x)}{d} \leq \frac{h(s_d)}{d} + \frac{\epsilon}{2}h_{\mathcal{N}_a}(x) + (t_1+4\epsilon)\Lambda_h + o(1).\]
Thanks to Proposition \ref{prop:siegel} and \eqref{eq:lambda-h-ineq}, it follows that
\begin{align*}
\frac{\epsilon}{2} h_{\mathcal{N}_a}(x) - h_{\mathcal{M}}(x) \leq \frac{2M'}{u\epsilon}(t_1+2t_2+\epsilon)\Lambda_h+(t_1+4\epsilon)\Lambda_h+o(1)Ê\\
\leq \left(\frac{2M'}{u\epsilon}\right)(2t_1+2t_2+5\epsilon)\Lambda_h +o(1)\\
\leq \left(\frac{M't_1}{\epsilon^2}\right)8(2+2t_2+5\epsilon)\Lambda^{2\psi(u)-2\psi(0)}(Mt_2)^{-u}\left(\frac{\epsilon}{4}\sum_{i=1}^{m}{a_ic^{(i)}_3} \right),
\end{align*}
where the strict inequality in \eqref{eq:lambda-h-ineq} allowed us to sweep the $o(1)$ under the rug (for $d$ large enough).

We have $8(2+2t_2+5\epsilon) \leq 42t_2 \leq \Lambda^2 t_2$ and it follows from (iii) that
\begin{align*}
(M't_1)\epsilon^{-2} \leq Mt_1\max_{i}{(N_i+1)}^{m}(2u\theta)^2(t_1m)^{2u}\left(\prod_{i=1}^{m}{\deg(X_i)}\right)^{2(1+\omega)}\Lambda^{2(1+\omega)(\psi(u)-1)}\\
\leq M\Lambda^{\max\{3,m\}+2(1+\omega)\psi(u)} \leq M\Lambda^{2\omega'u\psi(u)-2} = M\Lambda^{2\psi(u-1)-2\psi(u)-2},
\end{align*}
where we used that $\max\{3,m\}Ê\leq 4u\psi(u)-2$.

Hence, we can deduce that
\[\frac{\epsilon}{2} h_{\mathcal{N}_a}(x) - h_{\mathcal{M}}(x) \leqÊ(Mt_2)^{-(u-1)}\Lambda^{2\psi(u-1)-2\psi(0)}\frac{\epsilon}{4}\sum_{i=1}^{m}{a_ic^{(i)}_3}  \leq \frac{\epsilon}{4}h_{\mathcal{N}_a}(x),\]
from which it follows that $h_{\mathcal{N}_a}(x) \leq 4\epsilon^{-1}h_{\mathcal{M}}(x)$. The theorem follows, since by (iii)
\[4\epsilon^{-1} \leq 8u\theta(t_1m)^u\left(\prod_{i=1}^{m}{\deg(X_i)}\right)^{1+\omega}\Lambda^{(1+\omega)(\psi(u)-1)} \leq \Lambda^{(1+\omega)\psi(u)+2}\]
and $\Lambda^{(1+\omega)\psi(u)+2} \leq \Lambda^{\omega'u\psi(u)} \leq c_1$.

\section*{Acknowledgements}
I thank Thomas Ange for sharing his unpublished work. I thank my advisor Philipp Habegger for his continuous encouragement and for many helpful and interesting discussions. I thank Philipp Habegger and Ga\"{e}l R\'{e}mond for helpful comments on a preliminary version of this article. I thank the anonymous referee for their comments, which helped me to improve the exposition and led to a strengthening of the main result. This work was supported by the Swiss National Science Foundation as part of the project ``Diophantine Problems, o-Minimality, and Heights", no. 200021\_165525.

\bibliographystyle{acm}
\bibliography{Bibliography}

\end{document}